\documentclass[a4paper]{amsart}

\usepackage{amsmath,amssymb,amsthm}
\usepackage{eucal}
\usepackage[all]{xy}
\usepackage[colorlinks=true,backref=page]{hyperref}

\newtheorem{theorem}{Theorem}[section]
\newtheorem{proposition}[theorem]{Proposition}
\newtheorem{lemma}[theorem]{Lemma}

\theoremstyle{definition}
\newtheorem{definition}[theorem]{Definition}
\newtheorem{example}[theorem]{Example}

\newtheorem{problem}[theorem]{Problem}

\theoremstyle{remark}
\newtheorem{remark}[theorem]{Remark}

\numberwithin{equation}{section}

\newcommand{\Z}{\mathbb{Z}}

\newcommand{\N}{\mathtt{N}}
\newcommand{\open}{\operatorname{Open}}
\newcommand{\colim}{\operatorname{colim}}

\newcommand{\Ker}{\operatorname{Ker}}

\newcommand{\sky}{\mathtt{sky}}
\newcommand{\BM}{\mathrm{BM}}
\newcommand{\X}{\mathcal{X}}
\renewcommand{\P}{\mathcal{P}}

\SelectTips{cm}{}

\title[Cellular cosheaf homology are cosheaf homology]{Cellular cosheaf homology are cosheaf homology}

\author[Daisuke Kishimoto]{Daisuke Kishimoto}
\address{Department of Mathematics, Kyoto University, Kyoto, 606-8502, Japan}
\email{kishi@math.kyoto-u.ac.jp}

\author[Yasutomo Yushima]{Yasutomo Yushima}
\address{Department of Mathematics, Kyoto University, Kyoto, 606-8502, Japan}
\email{yyushima811@gmail.com}

\date{\today}

\subjclass[2010]{55N30, 18F20}

\keywords{cosheaf, cellular cosheaf, cosheafification, cosheaf homology, \v{C}ech homology, Borel-Moore homology, Alexandroff space}

\begin{document}

\maketitle

\begin{abstract}
  A cosheaf is the dual notion of a sheaf, but we cannot define its homology as the formal dual of sheaf cohomology, in general, because of the lack of the cosheafification. A cellular cosheaf is a contravariant functor from the face poset of a CW complex to the category of abelian groups. We show that given a cellular cosheaf $F$, there is a natural way to associate a cosheaf $\widehat{F}$, for which we can define homology as the formal dual of sheaf cohomology, such that the Borel-Moore homology of $F$ is isomorphic to the homology of $\widehat{F}$ whenever the underlying CW complex of $F$ is a simplicial complex.
\end{abstract}


\section{Introduction}\label{Introduction}

A \emph{cosheaf} is the dual notion of a sheaf, that is, a cosheaf is a covariant functor from the poset of open sets of a given topological space into the category of abelian groups, which preserves certain colimits. The first appearance of cosheaves dates back, at least, to the 1960's \cite{B,D}, and since then, they appear ubiquitously in broad areas of mathematics. However, we cannot define the homology of a cosheaf as the formal dual of sheaf cohomology as we cannot prove the existence of the cosheafification in general.

A \emph{cellular cosheaf} is a contravariant functor from the face poset of a CW complex to the category of abelian groups. The Borel-Moore homology can be defined for a cellular cosheaf, and is often called cellular cosheaf homology. Recently, cosheaves and cellular cosheaves are intensely studied in applied algebraic topology as they are used to describe time dependent change of properties of spaces. See \cite{C,C1,C2,C3,G,GH,M,N,R,SMP,Y,YG}. By definition, cellular cosheaves are not cosheaves, and so it is natural to pose:

\begin{problem}
  \label{problem}
  Given a cellular cosheaf $F$, is there a natural way to associate a cosheaf $\widehat{F}$ such that the Borel-Moore homology of $F$ and the homology of $\widehat{F}$, if it were defined, are isomorphic?
\end{problem}

It is well known that there is a one-to-one correspondence between posets and $T_0$-Alexandroff spaces. We use this correspondence to study Problem \ref{problem}, and so we slightly generalize cellular cosheaves such that a cellular cosheaf is defined to be a functor from a poset to the category of abelian groups. We will extend the above correspondence to that there is a one-to-one correspondence between equivalence classes of cellular cosheaves and cosheaves on $T_0$-Alexandroff space (Theorem \ref{correspondence}). Thus we study cosheaves on $T_0$-Alexandroff spaces. The cosheafification cannot be constructed, in general, simply by dualizing the sheafification as mentioned above. However, as for precosheaves over $T_0$-Alexandroff spaces, we can prove the following. See \cite{C,P} for other attempts to construct the cosheafification.

\begin{theorem}
  [Theorem \ref{cosheafification}]
  Each precosheaf on a $T_0$-Alexandroff space admits the cosheafification.
\end{theorem}

This enables us to construct a projective resolution for each cosheaf on a $T_0$-Alexandroff space (Proposition \ref{resolution}), so that we can define its homology as the formal dual of sheaf cohomology. Then we are able to compare the Borel-Moore homology of a cellular cosheaf and the homology of the corresponding cosheaf on a $T_0$-Alexandroff space. We will prove the following answer to Problem \ref{problem}.

\begin{theorem}
  [Theorem \ref{main}]
  \label{main intro}
  Let $F$ be a cellular cosheaf on a simplicial complex, and let $\widehat{F}$ be the corresponding cosheaf on a $T_0$-Alexandroff space. Then the Borel-Moore homology of $F$ is isomorphic to the homology of $\widehat{F}$.
\end{theorem}

We can import a plenty of results in sheaf theory into cosheaf homology if it were defined. Then Theorem \ref{main intro} would benefit especially applied algebraic topology as the Borel-Moore homology of cellular cosheaves are frequently used there.

To each poset $P$ we can associate a simplicial complex called an order complex, and as we will see in Section \ref{Cech homology and Borel-Moore homology}, we can also associate to a cellular cosheaf over a poset a cellular cosheaf on its order complex. Then we can define the Borel-Moore homology for a cellular cosheaf over a poset as the one for a cellular cosheaf on the associated order complex. We will see that this definition is consistent with that for a cellular cosheaf on a regular CW  (Lemma \ref{subdivision}). Then a result for a cellular cosheaf on regular CW complexes analogous to Theorem \ref{main intro} will be also obtained (Theorem \ref{regular}).


\subsection*{Acknowledgement}

The first author was partially supported by JSPS KAKENHI Grant Numbers JP17K05248 and JP19K03473.


\section{Cosheaves on $T_0$-Alexandroff spaces}\label{Cosheaves on Alexandroff spaces}

This section studies a correspondence between cellular cosheaves and cosheaves on $T_0$-Alexandroff spaces. This correspondence is based on the correspondence between posets and $T_0$-Alexandroff spaces, which we recall here. An \emph{Alexandroff space} is a topologuical space in which all intersections of open sets, possibly infinitely many, are open, and a $T_0$-space is a topological space such that given distinct two points, there is an open set which contains exactly one of the two points. Let $X$ be an Alexandroff space. Then for each $x$, there is the smallest open set containing $x$, which we denote by $U_x$. Namely, $U_x$ is the intersection of all open sets containing $x$. The $T_0$-property of an Alexandroff space can be characterized as follows \cite[Proposition 4]{S}.

\begin{lemma}
  \label{antisymmetry}
  An Alexandroff space is a $T_0$-space if and only if $U_x=U_y$ implies $x=y$.
\end{lemma}

Let $X$ be a $T_0$-Alexandroff space. We define a poset from $X$. For $x,y\in X$, we define $x\le y$ if $U_x\subset U_y$. Clearly, the relation $\le$ satisfies reflexivity and transitivity. By Lemma \ref{antisymmetry}, it also satisfies antisymmetry, and so the relation $\le$ is a partial order on $X$. Let $\P(X)$ denote the poset $(X,\le)$.

Next, we define a $T_0$-Alexandroff space from a poset. Let $P$ be a poset. Then there is a topology $\mathcal{O}$ on $P$ whose open basis consists of $P_{\ge x}=\{y\ge x\mid y\in P\}$ for $x\in P$. Let $\X(P)$ denote the topological space $(X,\mathcal{O})$.

\begin{lemma}
  For a poset $P$, $\X(P)$ is a $T_0$-Alexandroff space.
\end{lemma}

\begin{proof}
  Let $\{U_\lambda\}_{\lambda\in\Lambda}$ be a collection of open sets of $\X(P)$. Note that a subset $U\subset P$ is an open set of $\X(P)$ if and only if $x\in P$ belongs to $U$ whenever $x\ge y$ for some $y\in U$. If $x\in P$ satisfies $x\ge y$ for some $y\in\bigcap_{\lambda\in\Lambda}U_\lambda$, then $x\in U_\lambda$ for each $\lambda\in\Lambda$, implying $x\in\bigcap_{\lambda\in\Lambda}U_\lambda$. Thus $\bigcap_{\lambda\in\Lambda}U_\lambda$ is an open set of $\X(P)$, and so $\X(P)$ is an Alexandroff space. Clearly, $U_x=P_{\ge x}$ for each $x\in P$. Then $x\le y$ whenever $U_x\subset U_y$, implying $x=y$ whenever $U_x=U_y$. Thus by Lemma \ref{antisymmetry}, $\X(P)$ is a $T_0$-space, completing the proof.
\end{proof}

By definition, $\P(\X(P))=P$ for each poset $P$ and $\X(\P(X))=X$ for each $T_0$-Alexandroff space $X$. Then we obtain:

\begin{theorem}
  \label{poset-space}
  The assignments $\P$ and $\X$ give a one-to-one correspondence between posets and $T_0$-Alexandroff spaces.
\end{theorem}

\begin{remark}
  As in \cite{S}, a map $X\to Y$ between $T_0$-Alexandroff spaces is continuous if and only if the corresponding map $\P(X)\to\P(Y)$ is order-preserving. Then Theorem \ref{poset-space} can be extended to an isomorphism between the category of posets and the category of $T_0$-Alexandroff spaces.
\end{remark}

We define cellular cosheaves and cosheaves. Let $\mathbf{Ab}$ denote the category of abelian groups. We will always regard a poset $P$ as a category in which objects are elements of $P$ and there is a unique morphism $x\to y$ whenever $x\le y$.

\begin{definition}
  A \emph{cellular cosheaf} over a poset $P$ is a functor $F\colon P\to\mathbf{Ab}$.
\end{definition}

For a topological space $X$, let $\open(X)$ denote the poset of open sets of $X$. Let $\N(\mathcal{U})$ denote the nerve of an open cover $\mathcal{U}$.

\begin{definition}
  A \emph{precosheaf} on a topological space $X$ is a functor $F\colon\open(X)\to\mathbf{Ab}$. If the natural map
  \[
    \underset{\N(\mathcal{U})}{\colim}\,F\to F(U)
  \]
  is an isomorphism for each open set $U$ and each open cover $\mathcal{U}$ of $U$, then $F$ is called a \emph{cosheaf}.
\end{definition}

We extend the correspondence of Theorem \ref{poset-space} to cellular cosheaves and cosheaves. Let $F$ be a cellular cosheaf over a poset $P$. We define a precosheaf $\widehat{F}$ on $\X(P)$ by
\[
  \widehat{F}(U)=\underset{x\in U}{\colim}\,F(x)
\]
for each open set $U$ of $\X(P)$, where we regard $U$ as a subposet of $P$ on the RHS.

\begin{lemma}
  \label{|F|}
  Let $F$ be a cellular cosheaf over a poset $P$. Then $\widehat{F}$ is a cosheaf.
\end{lemma}

\begin{proof}
  Let $U$ be an open set of $\X(P)$, and let $\mathcal{U}$ be an open cover of $U$. By definition, $\underset{V\in\N(\mathcal{U})}{\colim}\underset{x\in V}{\colim}\,F(x)$ is the quotient
  \[
    \bigoplus_{x\in U}\bigoplus_{I_x}F(x)/\sim
  \]
  in which copies of $F(x)$ are identified and $a\sim b$ if $F(x<y)(a)=b$ for $a\in F(x)$ and $b\in F(y)$, where $I_x$ is the set of elements of $\N(\mathcal{U})$ including $U_x$. Then
  \[
    \underset{V\in\N(\mathcal{U})}{\colim}\underset{x\in V}{\colim}\,F(x)\cong\underset{x\in U}{\colim}\,F(x)
  \]
  and so we get
  \begin{equation}
    \label{hat-F}
    \underset{\N(\mathcal{U})}{\colim}\,\widehat{F}=\underset{V\in\N(\mathcal{U})}{\colim}\underset{x\in V}{\colim}\,F(x)\cong\underset{x\in U}{\colim}\,F(x)=\widehat{F}(U).
  \end{equation}
  Thus $\widehat{F}$ is a cosheaf, as stated.
\end{proof}

Let $P$ be a poset. Then there is an injective order-preserving map
\[
  \iota\colon P\to\open(\X(P)),\quad x\mapsto U_x.
\]
Note that for a cosheaf $F$ on $\X(P)$, $\iota^*(F)$ is a cellular cosheaf over a poset $P$. Let $\mathbf{Cosh}(X)$ denote the category of cosheaves on a topological space $X$. Now we are ready to prove:

\begin{theorem}
  \label{correspondence}
  For a poset $P$, the functor
  \[
    [P,\mathbf{Ab}]\to\mathbf{Cosh}(\X(P)),\quad F\mapsto\widehat{F}
  \]
  is an equivalence of categories.
\end{theorem}

\begin{proof}
  We show that the functor $\iota^*\colon\mathbf{Cosh}(\X(P))\to[P,\mathbf{Ab}]$ is the inverse of the functor $\widehat{\cdot}$. Let $F\in[P,\mathbf{Ab}]$, up to natural isomorphism. Since $x\in P$ is the terminal object of $U_x$,
  \[
    \iota^*(\widehat{F})(x)=\widehat{F}(U_x)=\underset{y\in U_x}{\colim}\,F(y)\cong F(x),
  \]
  implying $\iota^*\circ\widehat{\cdot}=1$. Let $G\in\mathbf{Cosh}(\X(P))$. For $U\in\open(\X(P))$, let $\mathcal{U}_U=\{U_x\}_{x\in U}$ be an open cover of $U$, and let $U_{x_0\ldots x_n}=U_{x_0}\cap\cdots\cap U_{x_n}$. Since $G$ is a cosheaf,
  \begin{align*}
    G(U)&\cong\underset{\N(\mathcal{U}_U)}{\colim}\,G\\
    &=\left(\bigoplus_{\substack{x_0,\ldots,x_{n_1}\in U\\n_1\ge 0}}G(U_{x_0\cdots x_{n_1}})\right)/\sim\\
    &=\left(\bigoplus_{x_\in U}G(U_{x})\bigoplus_{\substack{x_0,\ldots,x_{n_1}\in U\\n_1\ge 1}}G(U_{x_0\cdots x_{n_1}})\right)/\sim\\
    &\cong\left(\bigoplus_{x\in U}G(U_{x})\bigoplus_{\substack{x_0,\ldots,x_{n_1}\in U\\n_1\ge 1}}\underset{\N(\mathcal{U}_{U_{x_0\cdots x_{n_1}}})}{\colim}\,G\right)/\sim\\
    &\cong\left(\bigoplus_{x\in U}G(U_{x})\bigoplus_{\substack{x_0,\ldots,x_{n_1}\in U\\n_1\ge 1}}\bigoplus_{\substack{y_0,\ldots,y_{n_2}\in U_{x_0\cdots x_{n_1}}\\n_2\ge 0}}G(U_{y_0\cdots y_{n_2}})\right)/\sim\\
    &\cong\left(\bigoplus_{x\in U}G(U_{x})\bigoplus_{\substack{x_0,\ldots,x_{n_1}\in U\\n_1\ge 1}}\bigoplus_{y\in U_{x_0\cdots x_{n_1}}}G(U_y)\bigoplus_{\substack{y_0,\ldots,y_{n_2}\in U_{x_0\cdots x_{n_1}}\\n_2\ge 1}}\underset{\N(\mathcal{U}_{U_{y_0\cdots y_{n_2}}})}{\colim}\,G\right)/\sim\\
    &\;\;\vdots\\
    &\cong\left(\bigoplus_{x\in U}\bigoplus_{J_x}G(U_x)\right)/\sim
  \end{align*}
  where in the last term, copies of $G(U_x)$ are identified and $a\sim b$ if $G(U_x<U_y)(a)=b$ for $a\in G(U_x)$ and $b\in G(U_y)$. Then we get
  \begin{equation}
    \label{colim U_x}
    G(U)\cong\underset{x\in U}{\colim}\,G(U_x)
  \end{equation}
  and so we obtain
  \[
    \widehat{\iota^*(G)}(U)=\underset{x\in U}{\colim}\,\iota^*(G)(x)=\underset{x\in U}{\colim}\,G(U_x)\cong G(U),
  \]
  implying $\widehat{\cdot}\circ\iota^*=1$. Therefore the proof is complete.
\end{proof}


\section{Defining cosheaf homology}\label{Defining cosheaf homology}

This section proves that we can define the homology of a cosheaf on a $T_0$-Alexandroff space as the formal dual of sheaf cohomology. Let $F$ be a precosheaf on a topological space $X$. The \emph{costalk} of $F$ at $x\in X$ is defined by
\[
  F_x=\underset{\N(x)}{\lim}\,F,
\]
where $\N(x)$ denotes the subposet of $\open(X)$ consisting of all open sets containing $x$. If $X$ is an Alexandroff space, then $U_x$ is the minimal element of $\N(x)$, implying
\begin{equation}
  \label{costalk}
  F_x\cong F(U_x).
\end{equation}
Then we need not take limits to get costalks. This simple property will be a key in our study. There is a natural map
\begin{equation}
  \label{rho}
  \rho(U)\colon\underset{x\in U}{\colim}\,F(U_x)\to F(U)
\end{equation}
induced from the inclusions $U_x\to U$ for $x\in U$. By \eqref{colim U_x}, we have:

\begin{lemma}
  \label{cosheaf costalk}
  Let $F$ be a cosheaf on a $T_0$-Alexandroff space $X$. Then for each open set $U$ of $X$, the natural map \eqref{rho} is an isomorphism.
\end{lemma}

Let $F,G$ be cosheaves on a topological space $X$, and let $\alpha\colon F\to G$ be a homomorphism, that is, a natural transformation. Then there is the induced map $\alpha_x\colon F_x\to G_x$ for each $x\in X$. If $X$ is an Alexandroff space, then $\alpha_x=\alpha(U_x)$ by \eqref{costalk}. The following two lemmas will be quite useful.

\begin{lemma}
  \label{uniqueness}
  Let $F,G$ be cosheaves on a $T_0$-Alexandroff space $X$. For homomorphisms $\alpha,\beta\colon F\to G$, $\alpha=\beta$ if and only if $\alpha_x=\beta_x\colon F_x\to G_x$ for each $x\in X$.
\end{lemma}

\begin{proof}
  The if part is obvious, and so we prove the only if part. Let $U$ be an open set of $X$, and let $\mathcal{U}=\{U_x\}_{x\in U}$ be an open cover of $U$. 
  By Lemma \ref{cosheaf costalk}, the maps
  \[
    \underset{x\in U}{\colim}\,\alpha(U_x),\,\underset{x\in U}{\colim}\,\beta(U_x)\colon\underset{x\in U}{\colim}\,F(U_x)\to\underset{x\in U}{\colim}\,G(U_x)
  \]
  are identified with $\alpha(U)$ and $\beta(U)$, respectively. Thus since $\alpha_x=\alpha(U_x)$ and $\beta_x=\beta(U_x)$, we get $\alpha=\beta$ whenever $\alpha_x=\beta_x$ for each $x\in X$. Therefore the proof is finished.
\end{proof}

Now we define the cosheafification of a precosheaf as the formal dual of the sheafification. Let $F$ be a precosheaf on a topological space $X$. The \emph{cosheafification} of $F$ is a cosheaf $F^+$ on $X$ such that there is a homomorphism $\pi\colon F^+\to F$ satisfying that for any homomorphism $\alpha\colon G\to F$ for a cosheaf $G$, there is a unique homomorphism $\tilde{\alpha}\colon G\to F^+$ satisfying a commutative diagram
\[
  \xymatrix{
    &F^+\ar[d]^\pi\\
    G\ar[r]_\alpha\ar@/^15pt/[ru]^{\tilde{\alpha}}&F.
  }
\]

\begin{theorem}
  \label{cosheafification}
  Every precosheaf on a $T_0$-Alexandroff space admits the cosheafification.
\end{theorem}

\begin{proof}
  Let $F$ be a precosheaf on a $T_0$-Alexandroff space $X$. We define a precosheaf $F^+$ on $X$ by
  \[
    F^+(U)=\underset{x\in U}{\colim}\,F(U_x)
  \]
  for an open set $U$ of $X$. Let $\mathcal{U}$ be an open cover of an open set $U$ of $X$. As in \eqref{hat-F}, we have
  \[
    \underset{\N(\mathcal{U})}{\colim}\,F^+=\underset{V\in\N(\mathcal{U})}{\colim}\,F^+(V)=\underset{V\in\N(\mathcal{U})}{\colim}\underset{x\in V}{\colim}\,F(U_x)\cong\underset{x\in U}{\colim}\,F(U_x)=F^+(U).
  \]
  Then $F^+$ is a cosheaf. We define a homomorphism $\pi\colon F^+\to F$ by $\pi(U)=\rho(U)$, where $\rho(U)$ is as in \eqref{rho}. Let $\alpha\colon G\to F$ be a homomorphism from a cosheaf $G$. Then there is a commutative diagram
  \[
    \xymatrix{
      \underset{x\in U}{\colim}\,G(U_x)\ar[r]^{\underset{x\in U}{\colim}\,\alpha_x}\ar[d]_{\rho(U)}&\underset{x\in U}{\colim}\,F(U_x)\ar[d]^{\rho(U)}\ar@{=}[r]&F^+(U)\ar[d]^{\pi(U)}\\
      G(U)\ar[r]^{\alpha(U)}&F(U)\ar@{=}[r]&F(U),
    }
  \]
  where the maps $\rho$ are as in \eqref{rho}. The left $\rho(U)$ is an isomorphism by Lemma \ref{cosheaf costalk}, so that we can define a map
  \[
    \tilde{\alpha}(U)\colon G(U)\xrightarrow[\cong]{\rho(U)^{-1}}\underset{x\in U}{\colim}\,G(U_x)\xrightarrow{\underset{x\in U}{\colim}\,\alpha_x}\underset{x\in U}{\colim}\,F(U_x)=F^+(U)
  \]
  which is natural with respect to $U$. Thus we get a homomorphism $\tilde{\alpha}\colon G\to F^+$. Clearly, we have $\pi\circ\tilde{\alpha}=\alpha$. Since $\pi_x\colon F_x^+\to F_x$ is a natural isomorphism such that $\pi_x\circ\tilde{\alpha}_x=\alpha_x$, the uniqueness of $\tilde{\alpha}$ follows from Lemma \ref{uniqueness}, completing the proof.
\end{proof}

For a homomorphism $\alpha\colon F\to G$ of cosheaves on a topological space $X$, we define a precosheaf $\Ker\alpha$ on $X$ by
\[
  (\Ker\alpha)(U)=\Ker\{\alpha(U)\colon F(U)\to G(U)\}
\]
for an open set $U$ of $X$. The important point here is that the precosheaf $\Ker\alpha$ is not necessarily a cosheaf, even when $X$ is a $T_0$-Alexandroff space.

\begin{example}
  Let $X=\{a,b,c\}$ be a topological space whose open sets are $\emptyset,\{a\},\{a,b\},\{a,c\},X$. Clearly, $X$ is a $T_0$-Alexandroff space. Define a precosheaf $F,G$ on $X$ by
  \[
    F(U)=
    \begin{cases}
      \Z&U\ne\emptyset\\
      0&U=\emptyset
    \end{cases}
    \quad\text{and}\quad G(U)=
    \begin{cases}
      \Z&U=\{a\}\\
      0&U\ne\{a\},
    \end{cases}
  \]
  where the maps $F(U<V)$ are the identity map of $\Z$ for $U\ne\emptyset$. We can easily verify that $F,G$ are cosheaves. There is a unique surjective homomorphism $\alpha\colon F\to G$, and we have
  \[
    (\Ker\alpha)(U)=
    \begin{cases}
      \Z&U\ne\emptyset,\{a\}\\
      0&U=\emptyset,\{a\}.
    \end{cases}
  \]
  Then it follows that
  \[
    \underset{\N(\mathcal{U}_X)}{\colim}\Ker\alpha=\Z^2.
  \]
  On the other hand, $(\Ker\alpha)(X)=\Z$, and thus $\Ker\alpha$ is not a cosheaf.
\end{example}

\begin{remark}
  The precosheaf defined by using the cokernel is a cosheaf because colimits commute with colimits.
\end{remark}

We construct a projective resolution for a cosheaf on a $T_0$-Alexandroff space by dualizing the standard construction of an injective resolution for a sheaf. We say that a sequence of cosheaves
\[
  F\xrightarrow{\alpha}G\xrightarrow{\beta}H
\]
on a topological space $X$ is exact if $F_x\xrightarrow{\alpha_x}G_x\xrightarrow{\beta_x}H_x$ is an exact sequence of abelian groups for each $x\in X$. A resolution of a cosheaf $F$ is an exact sequence of cosheaves
\[
  \cdots\to R_n\to R_{n-1}\to\cdots\to R_0\to F\to 0
\]
which we denote by $R_\bullet\to F\to 0$, and a homotopy between resolutions for $F$ can be defined in the obvious way. We say that a homomorphism of cosheaves $\alpha\colon F\to G$ is surjective if $F\xrightarrow{\alpha}G\to 0$ is exact, and injective if $0\to F\xrightarrow{\alpha}G$ is exact. A cosheaf $P$ on a topological space $X$ is \emph{projective} if for any homomorphism $\alpha\colon P\to F$ and any surjective homomorphism $\beta\colon G\to F$, there is a commutative diagram
\[
  \xymatrix{
    &G\ar[d]^\beta\\
    P\ar[r]_\alpha\ar@/^15pt/[ru]&F.
  }
\]
A projective resolution for a cosheaf $F$ is a resolution $P_\bullet\to F\to 0$ such that each $P_n$ is projective.

We consider the formal dual of a skyscraper sheaf. Let $X$ be a topologiacal space. For $x\in X$ and an abelian group $A$, we define a precosheaf $\sky(x,A)$ by
\[
  \sky(x,A)(U)=
  \begin{cases}
    A&x\in U\\
    0&x\not\in U
  \end{cases}
\]
for an open set $U$ of $X$. Clearly, $\sky(x,A)$ is a cosheaf.

\begin{lemma}
  \label{skyscraper}
  Let $X$ be a topological space. For $x\in X$ and a projective abelian group $P$, $\sky(x,P)$ is a projective cosheaf.
\end{lemma}

\begin{proof}
  Let $F,G$ be cosheaves on $X$. Suppose that we are given a homomorphism $\alpha\colon\sky(x,P)\to F$ and a surjective homomorphism $\beta\colon G\to F$. Then since $P$ is projective and $\sky(x,P)_x\cong P$, there is a commutative diagram
  \[
    \xymatrix{
      &G_x\ar[d]^{\beta_x}\\
      \sky(x,P)_x\ar[r]_(.65){\alpha_x}\ar@/^15pt/[ru]^{\tilde{\alpha}_x}&F_x.
    }
  \]
  Let $U$ be an open set of $X$. For $x\in U$, let $\tilde{\alpha}(U)$ denote the composite
  \[
    \sky(x,P)(U)\cong\sky(x,P)_x\xrightarrow{\tilde{\alpha}_x}G_x\xrightarrow{\rm proj}G(U),
  \]
  and for $x\not\in U$, let $\tilde{\alpha}(U)=0\colon\sky(x,P)(U)=0\to G(U)$. Then we get a homomorphism $\tilde{\alpha}\colon\sky(x,P)\to G$ satisfying a commutative diagram
  \[
    \xymatrix{
      &G\ar[d]^{\beta}\\
      \sky(x,P)\ar[r]_(.65){\alpha}\ar@/^15pt/[ru]^{\tilde{\alpha}}&F.
    }
  \]
  Thus $\sky(x,P)$ is a projective cosheaf, as stated.
\end{proof}

\begin{proposition}
  [cf. {\cite[Chapter II, Theorem 3.2]{B1}}]
  \label{resolution}
  Every cosheaf on a $T_0$-Alexandroff space admits a projective resolution which is unique, up to homotopy.
\end{proposition}

\begin{proof}
  Let $F$ be a cosheaf on a $T_0$-Alexandroff space $X$. Let $\alpha\colon G\to F$ be a homomorphism of cosheaves. Then since $(\Ker\alpha)^+_x\cong(\Ker\alpha)_x$ for each $x\in X$, the composition of homomorphisms $(\Ker\alpha)^+\xrightarrow{\pi}\Ker\alpha\to G$ is injective. Thus it is sufficient to prove that there is a surjective homomorphism $P\to F$, where $P$ is a projective cosheaf.

  For each $x\in X$, there is a surjection $P_x\to F_x$, where $P_x$ is a projective abelian group. Then as in the proof of Lemma \ref{skyscraper}, we can define a homomorphism $\alpha(x)\colon\sky(x,P_x)\to F$ which is surjective on the costalks at $x$. Since the direct sum of projective cosheaves is a projective cosheaf, $\bigoplus_{x\in X}\sky(x,P_x)$ is a projective cosheaf by Lemma \ref{skyscraper}. Clearly, the homomorphism
  \[
    \bigoplus_{x\in X}\alpha(x)\colon\bigoplus_{x\in X}\sky(x,P_x)\to F
  \]
  is surjective. The uniqueness of a projective resolution, up to homotopy, can be proved by the standard argument in homological algebra, completing the proof.
\end{proof}

By Proposition \ref{resolution}, we are able to define the homology of a cosheaf on a $T_0$-Alexandroff space.

\begin{definition}
  Let $F$ be a cosheaf on a $T_0$-Alexandroff space $X$. The homology of $F$, denoted by $H_*(X;F)$, is defined by the homology of a chain complex
  \[
    \cdots\to P_n(X)\to P_{n-1}(X)\to\cdots\to P_0(X)\to 0,
  \]
  where $P_\bullet\to F\to 0$ is a projective resolution for $F$.
\end{definition}


\section{\v{C}ech homology and Borel-Moore homology}\label{Cech homology and Borel-Moore homology}

This section studies the \v{C}ech homology of a cosheaf on a $T_0$-Alexandroff space, and compare it with the Borel-Moore homology to prove Theorem \ref{main intro}.


\subsection{\v{C}ech homology}

Let $F$ be a precosheaf on a topological space $X$, and let $\mathcal{U}$ be an open cover of $X$. For $U_{\alpha_0},\ldots,U_{\alpha_n}\in\mathcal{U}$, let $U_{\alpha_0\cdots\alpha_n}=U_{\alpha_0}\cap\cdots\cap U_{\alpha_n}$. We define
\[
  \check{C}_n(\mathcal{U};F)=\bigoplus_{\alpha_0,\ldots,\alpha_n}F(U_{\alpha_0\cdots\alpha_n}),
\]
where the sum is taken over all non-empty $U_{\alpha_0\cdots\alpha_n}$ such that $\alpha_0,\ldots,\alpha_n$ are distinct. We define a map $\partial\colon\check{C}_n(\mathcal{U};F)\to\check{C}_{n-1}(\mathcal{U};F)$ by
\[
  \partial(x)=\sum_{i=0}^n(-1)^iF(U_{\alpha_0\cdots\alpha_n}<U_{\alpha_0\cdots\widehat{\alpha_i}\cdots\alpha_n})(x)
\]
for $x\in\check{C}_n(\mathcal{U};F)$. It is straightforward to check $\partial^2=0$. Then we get a chain complex $(\check{C}_*(\mathcal{U};F),\partial)$, and its homology is denoted by
\[
  \check{H}_*(\mathcal{U};F)
\]
and called the \emph{\v{C}ech homology} of $F$ with respect to $\mathcal{U}$. 



We compare \v{C}ech homology and cosheaf homology as well as \v{C}ech cohomology and sheaf cohomology. Let $\mathcal{U}$ be an open cover of a topological space $X$. For an open set $U$ of $X$, let
\[
  \mathcal{U}\vert_U=\{U\cap V\mid V\in\mathcal{U}\}
\]
be an open cover of $U$. For a precosheaf $F$, we can define a precosheaf $C_n(\mathcal{U};F)$ on $X$ by
\[
  C_n(\mathcal{U};F)(U)=\check{C}_n(\mathcal{U}\vert_U;F\vert_U)
\]
for an open set $U$ of $X$.

\begin{lemma}
  [cf. {\cite[Chapter III, Theorem 4.9]{B1}}]
  Let $F$ be a cosheaf on a $T_0$-Alexandroff space $X$, and let $\mathcal{U}$ be an open cover of $X$. Then the sequence
  \[
    C_\bullet(\mathcal{U};F)\to F\to 0
  \]
  is a resolution.
\end{lemma}

\begin{proof}
  For an open cover $\mathcal{V}$ of an open set $V$ of $X$, we have
  \begin{align*}
    \underset{\N(\mathcal{V})}{\colim}\,C_n(\mathcal{U};F)&=\underset{W\in\N(\mathcal{V})}{\colim}\bigoplus_{\alpha_0,\ldots,\alpha_n}F(U_{\alpha_0\cdots\alpha_n}\cap W)\\
    &\cong\bigoplus_{\alpha_0,\ldots,\alpha_n}\underset{W\in\N(\mathcal{V})}{\colim}F(U_{\alpha_0\cdots\alpha_n}\cap W)\\
    &\cong\bigoplus_{\alpha_0,\ldots,\alpha_n}F(U_{\alpha_0\cdots\alpha_n}\cap V)\\
    &=C_n(\mathcal{U};F)(V),
  \end{align*}
  implying $C_n(\mathcal{U};F)$ is a cosheaf. We can see that the sequence $C_\bullet(\mathcal{U};F)\to F\to 0$ is exact by substituting $U_x$, completing the proof.
\end{proof}



We dualize a flasque sheaf. A cosheaf $F$ on a topological space $X$ is called \emph{flasque} if $F(V<U)\colon F(V)\to F(U)$ is injective for all open sets $V\subset U$ of $X$. The following lemma follows at once from the definition of a flasque cosheaf.

\begin{lemma}
  [cf. {\cite[Chapter II, Theorem 5.5]{B1}}]
  \label{H(flasque)}
  If $F$ is a flasque cosheaf on a $T_0$-Alexandroff space $X$, then for $*\ge 1$,
  \[
    H_*(X;F)=0.
  \]
\end{lemma}

Now we compare \v{C}ech homology and cosheaf homology.

\begin{theorem}
  [cf. {\cite[Chapter III, Theorem 4.13]{B1}}]
  \label{comparison}
  Let $F$ be a cosheaf on a $T_0$-Alexandroff space $X$, and let $\mathcal{U}=\{U_\alpha\}_{\alpha\in A}$ be an open cover of $X$. If
  \[
    H_*(U_{\alpha_0\cdots\alpha_n};F\vert_{U_{\alpha_0\cdots\alpha_n}})=0
  \]
  for $*\ge 1$ and $n\ge 0$, then there is an isomorphism
  \[
    \check{H}_*(\mathcal{U};F)\cong H_*(X;F).
  \]
\end{theorem}

\begin{proof}
  Let $P_\bullet\to F\to 0$ be the projective resolution constructed in the proof of Proposition \ref{resolution}. Clearly, the direct sum of flasque cosheaves is flasque. Then since the cosheaf $\sky(x,A)$ is flasque, each $P_n$ is flasque. Thus $C_*(\mathcal{U};P_n)$ is flasque too. Consider a commutative diagram
  \[
    \xymatrix{
      &\vdots\ar[d]&\vdots\ar[d]&\vdots\ar[d]\\
      \cdots\ar[r]&\check{C}_1(\mathcal{U};P_1)\ar[r]\ar[d]&\check{C}_1(\mathcal{U};P_0)\ar[r]\ar[d]&\check{C}_1(\mathcal{U};F)\ar[r]\ar[d]&0\\
      \cdots\ar[r]&\check{C}_0(\mathcal{U};P_1)\ar[r]\ar[d]&\check{C}_0(\mathcal{U};P_0)\ar[r]\ar[d]&\check{C}_0(\mathcal{U};F)\ar[r]\ar[d]&0\\
      \cdots\ar[r]&P_1(X)\ar[r]\ar[d]&P_0(X)\ar[r]\ar[d]&F(X)\ar[r]\ar[d]&0\\
      &0&0&0.
    }
  \]
  By Lemma \ref{H(flasque)} and the assumption $H_*(U_{\alpha_0\cdots\alpha_n};F\vert_{U_{\alpha_0\cdots\alpha_n}})=0$ for $*\ge 1$ and $n\ge 0$, all sequences but the bottom and the right ones are exact. Then we obtain the isomorphism in the statement.
\end{proof}


\subsection{Borel-Moore homology}

Let $X$ be a CW complex, and let $P(X)$ denote the face poset of $X$. We abbreviate a cellular cosheaf over $P(X)^{\mathrm{op}}$ by a cellular cosheaf on $X$. Let $F$ be a cellular cosheaf on $X$. We set
\[
  C_n^\BM(X;F)=\bigoplus_{\substack{\sigma\in P(X)\\\dim\sigma=n}}F(\sigma)
\]
and define a map $\partial\colon C_n^\BM(X;F)\to C_{n-1}^\BM(X;F)$ by
\[
  \partial(x)=\sum_{\dim\tau=n-1}[\sigma:\tau]F(\sigma>\tau)(x)
\]
for $x\in F(\sigma)$ and $\sigma\in P_n(X)$ with $\dim\sigma=n$, where $[\sigma:\tau]$ denotes the coincidence number. Then it is straightforward to check $\partial^2=0$, so that we get a chain complex $(C_*^\BM(X;F),\partial)$. The \emph{Borel-Moore homology}
\[
  H^\BM_*(X;F)
\]
is defined to be the homology of the above chain complex. Now we prove:

\begin{theorem}
  \label{main}
  Let $F$ be a cellular cosheaf on a simplicial complex $K$. Then there is an isomorphism
  \[
    H^\BM_*(K;F)\cong H_*(\X(P(K));\widehat{F}).
  \]
\end{theorem}

\begin{proof}
  Let $V$ denote the vertex set of $K$, and let $\mathcal{U}=\{U_v\}_{v\in V}$ be an open cover of $\X(P(K))$. If $U_{v_0}\cap\cdots\cap U_{v_n}\ne\emptyset$ for distinct $v_0,\ldots,v_n$, then vertices $v_0,\ldots,v_n$ form a simplex $[v_0,\ldots,v_n]$ of $K$ such that $U_{v_0}\cap\cdots\cap U_{v_n}=U_{[v_0,\ldots,v_n]}$. Thus $(\check{C}_*(\mathcal{U};\widehat{F}),\partial)\cong(C^\BM_*(K;F),\partial)$, implying
  \begin{equation}
    \label{Cech-BM}
    \check{H}_*(\mathcal{U};\widehat{F})\cong H^\BM_*(K;F).
  \end{equation}
  Let $P_\bullet\to\widehat{F}\to 0$ be a projective resolution for $\widehat{F}$. Then by \eqref{costalk}, the chain complex $\cdots\to P_n(U_{[v_0,\ldots,v_n]})\to P_{n-1}(U_{[v_0,\ldots,v_n]})\to\cdots$ is identified with a sequence of costalks $\cdots\to(P_n)_{[v_0,\ldots,v_n]}\to(P_{n-1})_{[v_0,\ldots,v_n]}\to\cdots$ which is exact by definition. Then we get 
  \[
    H_*(U_{v_0}\cap\cdots\cap U_{v_n};\widehat{F}\vert_{U_{v_0}\cap\cdots\cap U_{v_n}})=H_*(U_{[v_0,\ldots,v_n]};\widehat{F}\vert_{U_{[v_0,\ldots,v_n]}})=0
  \]
  for $*\ge 1$. Thus by Theorem \ref{comparison}, there is an isomorphism
  \begin{equation}
    \label{Cech-homology}
    \check{H}_*(\mathcal{U};\widehat{F})\cong H_*(\X(P(K));\widehat{F}).
  \end{equation}
  Therefore by combining \eqref{Cech-BM} and \eqref{Cech-homology}, the proof is finished.
\end{proof}

We discuss the Borel-Moore homology of a general cellular cosheaf. Let $P$ be a poset. Recall that the order complex $\Delta(P)$ is the geometric realization of an abstract simplicial complex whose simplices are finite sequences $x_0<\cdots<x_n$ in $P$. It is well known that $X\cong\Delta(P(X))$ whenever $X$ is a regular CW complex, and a relation between functors $\Delta$ and $\X$ can be found in \cite[Section 1]{Bar}. Let $F$ be a cellular cosheaf over a poset $P$. Then we can define a cellular cosheaf $\Delta(F)$ on $\Delta(P)$ by
\[
  \Delta(F)(x_0<\cdots<x_n)=F(x_n).
\]
Then we can define the Borel-Moore homology of $F$ by $H^\BM_*(\Delta(P);\Delta(F))$, which is also known as the higher colimits of $F$ (see \cite{BK}). Thus by Theorem \ref{main}, the Borel-Moore homology of $F$ is also identified with the homology of the naturally associated cosheaf. The above definition of the Borel-Moore homology for a functor over a general poset is consistent with the case of a functor over the face poset of a regular CW complex by the following dual version of \cite[Theorem 11.27]{C}.

\begin{lemma}
  \label{subdivision}
  Let $F$ be a cellular cosheaf on a regular CW complex $X$. Then there is an isomorphism
  \[
    H_*^\BM(X;F)\cong H_*^\BM(\Delta(P(X));\Delta(F)).
  \]
\end{lemma}

Thus by Theorem \ref{main}, we obtain:

\begin{theorem}
  \label{regular}
  Let $F$ be a cellular cosheaf on a regular CW complex $X$. Then there is an isomorphism
  \[
    H_*^\BM(X;F)\cong H_*(\X(P(\Delta(P(X))));\widehat{\Delta(F)}).
  \]
\end{theorem}

\end{document}